\documentclass[12pt]{amsart}
\usepackage{amsmath, amsthm, amssymb}
\usepackage[mathscr]{eucal}
\usepackage{fullpage}
\usepackage{enumerate}
\usepackage[all]{xy}
\usepackage{pdfsync}
\usepackage{pgf, tikz}

\newtheorem{theorem}{Theorem}[section]

\newtheorem{lemma}[theorem]{Lemma}
\newtheorem{corollary}[theorem]{Corollary}
\newtheorem{corollary*}{Corollary}
\newtheorem{theorem*}{Theorem}

\theoremstyle{definition}
\newtheorem{definition}[theorem]{Definition}

\newcommand{\pic}{\begin{tikzpicture}}
\newcommand{\epic}{\end{tikzpicture}}

\newcommand{\N}{\mathbb{N}}

\newcommand{\pf}{\begin{proof}}
\newcommand{\epf}{\end{proof}}
\newcommand{\Tab}{\mbox{Tab}}
\newcommand{\Red}{\mbox{Red}}
\newcommand{\LS}{\mbox{LS}}
\newcommand{\EG}{\mbox{EG}}

\begin{document}

\begin{abstract}
The Little map and the Edelman-Greene insertion algorithm, a generalization of the Robinson-Schensted correspondence, are both used for enumerating the reduced decompositions of an element of the symmetric group. 
We show the Little map factors through Edelman-Greene insertion and establish new results about each map as a consequence. 
In particular, we resolve some conjectures of Lam and Little.
\end{abstract}

\title{Relating Edelman-Greene insertion to the
Little map}
\author{Zachary Hamaker and Benjamin Young}
\date{\today}
\maketitle

\section{Introduction}
\label{sec:introduction}

\subsection{Preliminaries}
\label{ssec:Preliminaries}
In this paper, we clarify the relationship between two algorithmic bijections, due respectively to Edelman-Greene~\cite{edelman1987balanced} and to Little~\cite{little2003combinatorial}, both of which deal with reduced decompositions in the symmetric group $S_n$.  It is well known that $S_n$ can be  viewed as a Coxeter group with the presentation
\[
S_n = \langle s_1,s_2, \dots, s_{n-1} \mid s_i^2 = 1,\ s_i s_j = s_j s_i\ \mbox{for} \ |i-j| \geq 2,\ s_i s_{i+1} s_i = s_{i+1} s_i s_{i+1}\rangle
\]
where $s_i$ can be viewed as the transposition $(i\ i{+}1)$. Let $\sigma = \sigma_1 \sigma_2 \dots \sigma_n \in S_n$. 
A \emph{reduced decomposition} or \emph{reduced expression} of $\sigma$ is a minimal-length sequence $s_{w_1}, s_{w_2}, \dots, s_{w_m}$ such that $\sigma = s_{w_1} s_{w_2} \dots s_{w_m}$.  
The word $w = w_1w_2\dots w_m$ is called a \emph{reduced word} of $\sigma$. 
It is convenient to refer to a reduced decomposition by its corresponding reduced word and we will conflate the two often. 
The set of all reduced decompositions of $\sigma$ is denoted $\Red(\sigma)$. 
An \emph{inversion} in $\sigma$ is a pair $(i,j)$ with $i<j$ and $\sigma_i>\sigma_j$.
Let $l(\sigma)$ be the number of inversions in $\sigma$. 
Since each transposition $s_i$ either introduces or removes an inversion, for $w = w_1 \dots w_m$ a reduced word of $\sigma$, we can show $m  = l(\sigma)$.

The enumerative theory of reduced decompositions was first studied in~\cite{stanley1984number}, where using algebraic techniques it is shown for the reverse permutation $\sigma = n \dots 21$ that
\begin{equation}
\label{eq:count sorting networks}
|\Red(\sigma)| = \frac{{n \choose 2}!}{(2n-3)(2n-5)^2\dots5^{n-2}3^{n-2}}.
\end{equation}
This is the same as the number of standard Young tableaux with the staircase shape $\lambda = (n-1,n-2, \dots , 1)$. 
In addition, Stanley conjectured  for arbitrary $\sigma \in S_n$ that $|\Red(\sigma)|$ can be expressed as the number of standard Young tableaux of various shapes (possibly with multiplicity). 
This conjecture was resolved in~\cite{edelman1987balanced} using a generalization of the Robinson-Schensted insertion algorithm, usually called \emph{Edelman-Greene insertion}. 
Edelman-Greene insertion maps a reduced word $w$ to the pair of Young tableaux $(P(w),Q(w))$ where the entries of $P(w)$ are row-and-column strict and $Q(w)$ is a standard Young tableau.  
The same map also provides a bijective proof of (\ref{eq:count sorting networks}), as there is only one possibility for $P(w)$ while every standard $Q(w)$ is possible. 

Algebraic techniques developed in~\cite{lascoux1985schubert} can be used to compute the exact multiplicity of each shape for given $\sigma$. 
A bijective realization of Lascoux and Sch\"utzenberger's techniques in this setting is demonstrated in~\cite{little2003combinatorial}. 
A \emph{descent} is an inversion of the form $(i,i+1)$.
Permutations with precisely one descent are referred to as \emph{Grassmannian}.
There is a simple bijection between reduced words of a Grassmannian permutation $\sigma$ and standard Young tableaux of a shape determined by $\sigma$.
The Little map works by applying a sequence of modifications referred to as \emph{Little bumps} to the reduced word $w$ until the modified word's corresponding permutation is Grassmannian so that it can be mapped to a standard Young tableau denoted $\LS(w)$. 

\subsection{Results}
\label{ssec:results}
Since the Little map's introduction, there has been speculation on its relationship to Edelman-Greene insertion. 
In the appendix of~\cite{garsia2002saga}, written by Little, Conjecture 4.3.2 asserts that $\LS(w) = Q(w)$ when the maps are restricted to reduced words of the reverse permutation. 
Similar comments are made in~\cite{little2003combinatorial}.
We show the connection is much stronger than previously suspected: this equality is true for every permutation.
\begin{theorem}
\label{thm:same map}
Let $w$ be a reduced word. Then
\[
Q(w) = \LS(w).
\]
\end{theorem}
The proof is based on an argument from canonical form. 
First, we verify the theorem for the \emph{column reading word}, a canonical reduced word associated to $P(w)$ that plays nicely with both Edelman-Greene insertion and Little bumps.
We then show the statement's truth is invariant under Coxeter-Knuth moves, transformations that span the space of reduced words with identical $P(w)$.

Given Theorem \ref{thm:same map}, one might suspect the respective structures of the two maps are intimately related.  
Specifically, Conjecture 2.5 of~\cite{lam2010stanley} proposes that Little bumps relate to Edelman-Greene insertion in a way that is analogous to the role dual Knuth transformations play for the Robinson-Schensted-Knuth algorithm.

Let $v$ and $w$ be reduced words.
We say $v$ and $w$ \emph{communicate} if there exists a sequence of Little bumps changing $v$ to $w$.
This is an equivalence relation as Little bumps are invertible.
\begin{theorem}[Lam's Conjecture]
\label{thm:lam}

Let $v$ and $w$ be two reduced words. Then $v$ and $w$ communicate if and only if $Q(v) = Q(w)$.

\end{theorem}

\subsection{Random Sorting networks}
In recent years, there has been interest in the properties of randomly chosen reduced decompositions for the reverse permutation, known as \emph{random sorting networks}.
 introduced in~\cite{angel2007random} and studied further in~\cite{angel2011pattern, angel2010random}.  
Little is known about these objects rigorously, though conjectures are plentiful, striking, and strongly supported by numerical evidence.  
Most of the results that are known come from analyzing the asymptotics of staircase-shaped Young tableaux, by way of the Edelman-Greene correspondence.  

For instance, it is conjectured in~\cite[Conjecture 2]{angel2007random} that the ``partial'' permutation matrix of a random sorting network, obtained by concatenating the first half of the transpositions, has its nonzero entries distributed according to the Archimedean distribution.
Curiously, this distribution is also found in the limiting shape of a random domino tiling of the Aztec Diamond~\cite{chhita2012asymptotic}.  
However, the current best result in this direction~\cite[Theorem 4]{angel2007random} is an octagonal bound on the support of the nonzero entries in the partial permutation matrix.
The bound is obtained by computing the limiting profile of a random staircase Young tableau (which can be done precisely), and then attempting to push this information through the Edelman-Greene correspondence (which cannot).  

We think that an incomplete understanding of the Edelman-Greene algorithm is one of the main obstacles to progress on the random sorting network problem.  
As such, we hope that by strengthening the combinatorial foundations of this area, better asymptotic characterizations of random sorting networks will be attained.  
We regard this paper as a first step in this direction.

\subsection{Structure of the paper}
\label{ssec:structure}

In the second section, we review those parts of~\cite{edelman1987balanced,little2003combinatorial} which we need: we define Edelman-Greene insertion and the Little map, as well as generalized Little bumps.
Additionally, we state some properties of these maps that are important to our work.
The third section defines Coxeter-Knuth transformations and studies their interaction with Little bumps and action on $Q(w)$.
We conclude in the  fourth section by proving our main results and resolving several conjectures of Little.

\subsection{Acknowledgments}
We would like to thank Omar Angel, Vadim Gorin, Ander Holroyd, Thomas Lam, David Little, Eric Nordenstam, Dan Romik, Balint Virag and Peter Winkler for helpful discussions.  This research began while we were visiting the Mathematical Sciences Research Institute in 2012, for the program in Random Spatial Processes.

This project made heavy use of computer experiments in Sage~\cite{sage}.  Also, David Little's applet~\cite{littleapplet} was very helpful in understanding the Little bijection.
 
\section{Two Maps}
\label{sec:two maps}

\subsection{Edelman-Greene insertion}
\label{ssec:eg}
In order to define Edelman-Greene insertion, we must first define a rule for inserting a number into a tableau.
Let $n \in \N$ and $T$ be a tableau with rows $R_1, R_2, \dots, R_k$ where $R_i = r^i_1\leq r^i_2\leq \dots\leq r^i_{l_i}$. 
We define the insertion rule for Edelman-Greene insertion, following~\cite{edelman1987balanced}.
\begin{enumerate}
\item If $n \geq r^1_{l_1}$ or if $R_1$ is empty, adjoin $n$ to the end of $R_1$.
\item  If $n < r^1_{l_1}$, let $j$ be the smallest number such that $n < r^1_j$. 
\begin{enumerate}
\item If $r^1_j = n+1$ and $r^1_{j-1} = n$, insert $n+1$ into $T' = R_2,\dots, R_k$ and leave $R_1$ unchanged.
\item Otherwise, replace $r^1_j$ with $n$ and insert it into $T' = R_2, \dots, R_k$.
\end{enumerate}
\end{enumerate}  
Aside from 2(a), this is the RSK insertion rule.
Such exceptional bumps are referred to as \emph{special}.
For $w = w_1 \dots w_m$ a word (not necessarily reduced), we define $\EG(w) = (P(w),Q(w))$ via the following sequence of tableaux (see Figure 
\ref{fig:eg example} for an example).
We obtain $P_1(w)$ by inserting $w_m$ into the empty tableau. 
Then $P_j(w)$ is obtained by inserting $w_{m-j+1}$ into $P_{j-1}(w)$.
Note we are inserting the entries of $w$ from right to left.
At each step, one additional box is added.
In $Q(w)$, the entry of each box records the time of the step in which it was added.
From this, we can conclude that $Q(w)$ is a standard Young tableau.
Note the fourth insertion in Figure \ref{fig:eg example} follows 2(a).
For $w$ a reduced word of some $\sigma$, it is shown in~\cite{edelman1987balanced} that the entries of $P(w)$ are strictly increasing across rows and down columns.
Additionally, we can recover $\sigma$ from $P(w)$ with no additional information, that is $P(w)$ determines $\sigma$.
\begin{figure}
\caption{Edelman-Greene insertion for $w = 4,2,1,2,3,2,4$
\label{fig:eg example}}

\begin{center}
\pic[scale = .5]

\draw (1,0) -- (0,0) -- (0,-1);
\draw (0,-1) -- (1,-1) -- (1,0);
\node at (0.5,-0.5) (1/1) {4};
\node at (1,1) {$P_1$};

\epic
\hspace{12pt}
\pic[scale = .5]

\draw (1,0) -- (0,0) -- (0,-1);
\draw (0,-1) -- (1,-1) -- (1,0);
\node at (0.5,-0.5) (1/1) {1};
\node at (1,1) {$Q_1$};

\epic
\hspace{24pt}
\pic[scale = .5]

\draw (1,0) -- (0,0) -- (0,-2);
\draw (0,-1) -- (1,-1) -- (1,0);
\node at (0.5,-0.5) (1/1) {2};
\draw (0,-2) -- (1,-2) -- (1,-1);
\node at (0.5,-1.5) (1/2) {4};
\node at (1,1) {$P_2$};

\epic
\hspace{12pt}
\pic[scale = .5]

\draw (1,0) -- (0,0) -- (0,-2);
\draw (0,-1) -- (1,-1) -- (1,0);
\node at (0.5,-0.5) (1/1) {1};
\draw (0,-2) -- (1,-2) -- (1,-1);
\node at (0.5,-1.5) (1/2) {2};
\node at (1,1) {$Q_2$};

\epic
\hspace{24pt}
\pic[scale = .5]

\draw (2,0) -- (0,0) -- (0,-2);
\draw (0,-1) -- (1,-1) -- (1,0);
\node at (0.5,-0.5) (1/1) {2};
\draw (1,-1) -- (2,-1) -- (2,0);
\node at (1.5,-0.5) (2/1) {3};
\draw (0,-2) -- (1,-2) -- (1,-1);
\node at (0.5,-1.5) (1/2) {4};
\node at (2,1) {$P_3$};

\epic
\hspace{12pt}
\pic[scale = .5]

\draw (2,0) -- (0,0) -- (0,-2);
\draw (0,-1) -- (1,-1) -- (1,0);
\node at (0.5,-0.5) (1/1) {1};
\draw (1,-1) -- (2,-1) -- (2,0);
\node at (1.5,-0.5) (2/1) {3};
\draw (0,-2) -- (1,-2) -- (1,-1);
\node at (0.5,-1.5) (1/2) {2};
\node at (2,1) {$Q_3$};

\epic
\hspace{24pt}
\pic[scale = .5]

\draw (2,0) -- (0,0) -- (0,-3);
\draw (0,-1) -- (1,-1) -- (1,0);
\node at (0.5,-0.5) (1/1) {2};
\draw (1,-1) -- (2,-1) -- (2,0);
\node at (1.5,-0.5) (2/1) {3};
\draw (0,-2) -- (1,-2) -- (1,-1);
\node at (0.5,-1.5) (1/2) {3};
\draw (0,-3) -- (1,-3) -- (1,-2);
\node at (0.5,-2.5) (1/3) {4};
\node at (2,1) {$P_4$};

\epic
\hspace{12pt}
\pic[scale = .5]

\draw (2,0) -- (0,0) -- (0,-3);
\draw (0,-1) -- (1,-1) -- (1,0);
\node at (0.5,-0.5) (1/1) {1};
\draw (1,-1) -- (2,-1) -- (2,0);
\node at (1.5,-0.5) (2/1) {3};
\draw (0,-2) -- (1,-2) -- (1,-1);
\node at (0.5,-1.5) (1/2) {2};
\draw (0,-3) -- (1,-3) -- (1,-2);
\node at (0.5,-2.5) (1/3) {4};
\node at (2,1) {$Q_4$};

\epic
\end{center}
\medskip
\begin{center}
\pic[scale = .5]

\draw (2,0) -- (0,0) -- (0,-4);
\draw (0,-1) -- (1,-1) -- (1,0);
\node at (0.5,-0.5) (1/1) {1};
\draw (1,-1) -- (2,-1) -- (2,0);
\node at (1.5,-0.5) (2/1) {3};
\draw (0,-2) -- (1,-2) -- (1,-1);
\node at (0.5,-1.5) (1/2) {2};
\draw (0,-3) -- (1,-3) -- (1,-2);
\node at (0.5,-2.5) (1/3) {3};
\draw (0,-4) -- (1,-4) -- (1,-3);
\node at (0.5,-3.5) (1/4) {4};
\node at (2,1) {$P_5$};

\epic
\hspace{12pt}
\pic[scale = .5]

\draw (2,0) -- (0,0) -- (0,-4);
\draw (0,-1) -- (1,-1) -- (1,0);
\node at (0.5,-0.5) (1/1) {1};
\draw (1,-1) -- (2,-1) -- (2,0);
\node at (1.5,-0.5) (2/1) {3};
\draw (0,-2) -- (1,-2) -- (1,-1);
\node at (0.5,-1.5) (1/2) {2};
\draw (0,-3) -- (1,-3) -- (1,-2);
\node at (0.5,-2.5) (1/3) {4};
\draw (0,-4) -- (1,-4) -- (1,-3);
\node at (0.5,-3.5) (1/4) {5};
\node at (2,1) {$Q_5$};

\epic
\hspace{24pt}
\pic[scale = .5]

\draw (2,0) -- (0,0) -- (0,-4);
\draw (0,-1) -- (1,-1) -- (1,0);
\node at (0.5,-0.5) (1/1) {1};
\draw (1,-1) -- (2,-1) -- (2,0);
\node at (1.5,-0.5) (2/1) {2};
\draw (0,-2) -- (1,-2) -- (1,-1);
\node at (0.5,-1.5) (1/2) {2};
\draw (1,-2) -- (2,-2) -- (2,-1);
\node at (1.5,-1.5) (2/2) {3};
\draw (0,-3) -- (1,-3) -- (1,-2);
\node at (0.5,-2.5) (1/3) {3};
\draw (0,-4) -- (1,-4) -- (1,-3);
\node at (0.5,-3.5) (1/4) {4};
\node at (2,1) {$P_6$};

\epic
\hspace{12pt}
\pic[scale = .5]

\draw (2,0) -- (0,0) -- (0,-4);
\draw (0,-1) -- (1,-1) -- (1,0);
\node at (0.5,-0.5) (1/1) {1};
\draw (1,-1) -- (2,-1) -- (2,0);
\node at (1.5,-0.5) (2/1) {3};
\draw (0,-2) -- (1,-2) -- (1,-1);
\node at (0.5,-1.5) (1/2) {2};
\draw (1,-2) -- (2,-2) -- (2,-1);
\node at (1.5,-1.5) (2/2) {6};
\draw (0,-3) -- (1,-3) -- (1,-2);
\node at (0.5,-2.5) (1/3) {4};
\draw (0,-4) -- (1,-4) -- (1,-3);
\node at (0.5,-3.5) (1/4) {5};
\node at (2,1) {$Q_6$};

\epic
\hspace{24pt}
\pic[scale = .5]

\draw (3,0) -- (0,0) -- (0,-4);
\draw (0,-1) -- (1,-1) -- (1,0);
\node at (0.5,-0.5) (1/1) {1};
\draw (1,-1) -- (2,-1) -- (2,0);
\node at (1.5,-0.5) (2/1) {2};
\draw (2,-1) -- (3,-1) -- (3,0);
\node at (2.5,-0.5) (3/1) {4};
\draw (0,-2) -- (1,-2) -- (1,-1);
\node at (0.5,-1.5) (1/2) {2};
\draw (1,-2) -- (2,-2) -- (2,-1);
\node at (1.5,-1.5) (2/2) {3};
\draw (0,-3) -- (1,-3) -- (1,-2);
\node at (0.5,-2.5) (1/3) {3};
\draw (0,-4) -- (1,-4) -- (1,-3);
\node at (0.5,-3.5) (1/4) {4};
\node at (3,1) {$P_7=P(w)$};

\epic
\hspace{12pt}
\pic[scale = .5]

\draw (3,0) -- (0,0) -- (0,-4);
\draw (0,-1) -- (1,-1) -- (1,0);
\node at (0.5,-0.5) (1/1) {1};
\draw (1,-1) -- (2,-1) -- (2,0);
\node at (1.5,-0.5) (2/1) {3};
\draw (2,-1) -- (3,-1) -- (3,0);
\node at (2.5,-0.5) (3/1) {7};
\draw (0,-2) -- (1,-2) -- (1,-1);
\node at (0.5,-1.5) (1/2) {2};
\draw (1,-2) -- (2,-2) -- (2,-1);
\node at (1.5,-1.5) (2/2) {6};
\draw (0,-3) -- (1,-3) -- (1,-2);
\node at (0.5,-2.5) (1/3) {4};
\draw (0,-4) -- (1,-4) -- (1,-3);
\node at (0.5,-3.5) (1/4) {5};
\node at (3,1) {$Q_7=Q(w)$};

\epic
\end{center}
\end{figure}

\begin{figure}
\caption{The Little map for the reduced decomposition $w_4 w_2 w_1 w_2 w_3 w_2 w_4$ of $\sigma = 35241$.  The dashed crosses show the modifications made by the next Little bump.
\label{fig:little example}}

\begin{tabular}{cc}
\pic[scale = .85]
\node[draw,circle] at (0,0) (0/1) {1};
\node[draw,circle] at (0,-1) (0/2) {2};
\node[draw,circle] at (0,-2) (0/3) {3};
\node[draw,circle] at (0,-3) (0/4) {4};
\node[draw,circle] at (0,-4) (0/5) {5};

\node[draw,circle] at (1,0) (1/1) {1};
\node[draw,circle] at (1,-1) (1/2) {2};
\node[draw,circle] at (1,-2) (1/3) {3};
\node[draw,circle] at (1,-3) (1/4) {5};
\node[draw,circle] at (1,-4) (1/5) {4};

\node[draw,circle] at (2,0) (2/1) {1};
\node[draw,circle] at (2,-1) (2/2) {3};
\node[draw,circle] at (2,-2) (2/3) {2};
\node[draw,circle] at (2,-3) (2/4) {5};
\node[draw,circle] at (2,-4) (2/5) {4};

\node[draw,circle] at (3,0) (3/1) {3};
\node[draw,circle] at (3,-1) (3/2) {1};
\node[draw,circle] at (3,-2) (3/3) {2};
\node[draw,circle] at (3,-3) (3/4) {5};
\node[draw,circle] at (3,-4) (3/5) {4};

\node[draw,circle] at (4,0) (4/1) {3};
\node[draw,circle] at (4,-1) (4/2) {2};
\node[draw,circle] at (4,-2) (4/3) {1};
\node[draw,circle] at (4,-3) (4/4) {5};
\node[draw,circle] at (4,-4) (4/5) {4};

\node[draw,circle] at (5,0) (5/1) {3};
\node[draw,circle] at (5,-1) (5/2) {2};
\node[draw,circle] at (5,-2) (5/3) {5};
\node[draw,circle] at (5,-3) (5/4) {1};
\node[draw,circle] at (5,-4) (5/5) {4};

\node[draw,circle] at (6,0) (6/1) {3};
\node[draw,circle] at (6,-1) (6/2) {5};
\node[draw,circle] at (6,-2) (6/3) {2};
\node[draw,circle] at (6,-3) (6/4) {1};
\node[draw,circle] at (6,-4) (6/5) {4};

\node[draw,circle] at (7,0) (7/1) {3};
\node[draw,circle] at (7,-1) (7/2) {5};
\node[draw,circle] at (7,-2) (7/3) {2};
\node[draw,circle] at (7,-3) (7/4) {4};
\node[draw,circle] at (7,-4) (7/5) {1};

\draw[ultra thick] (0/1) -- (1/1) -- (2/1) -- (3/2) -- (4/3) -- (5/4) -- (6/4) -- (7/5);
\draw[ultra thick] (0/2) -- (1/2) -- (2/3) -- (3/3) -- (4/2) -- (5/2) -- (6/3) -- (7/3);
\draw[ultra thick] (0/3) -- (1/3) -- (2/2) -- (3/1) -- (4/1) -- (5/1) -- (6/1) -- (7/1);
\draw[ultra thick] (0/4) -- (1/5) -- (2/5) -- (3/5) -- (4/5) -- (5/5) -- (6/5) -- (7/4);
\draw[ultra thick] (0/5) -- (1/4) -- (2/4) -- (3/4) -- (4/4) -- (5/3) -- (6/2) -- (7/2);

\draw[ultra thick, dashed] (6/4) -- (7/3); 
\draw[ultra thick, dashed] (6/3) -- (7/4); 

\draw[ultra thick, dashed] (3/2) -- (4/1); 
\draw[ultra thick, dashed] (3/1) -- (4/2); 

\draw[ultra thick, dashed] (2/1) -- (2.75,.75); 
\draw[ultra thick, dashed] (2.25,.75) -- (3/1); 

\epic
&
\pic[scale = .85]
\node[draw,circle] at (0,0) (0/1) {1};
\node[draw,circle] at (0,-1) (0/2) {2};
\node[draw,circle] at (0,-2) (0/3) {3};
\node[draw,circle] at (0,-3) (0/4) {4};
\node[draw,circle] at (0,-4) (0/5) {5};
\node[draw,circle] at (0,-5) (0/6) {6};

\node[draw,circle] at (1,0) (1/1) {1};
\node[draw,circle] at (1,-1) (1/2) {2};
\node[draw,circle] at (1,-2) (1/3) {3};
\node[draw,circle] at (1,-3) (1/4) {4};
\node[draw,circle] at (1,-4) (1/5) {6};
\node[draw,circle] at (1,-5) (1/6) {5};

\node[draw,circle] at (2,0) (2/1) {1};
\node[draw,circle] at (2,-1) (2/2) {2};
\node[draw,circle] at (2,-2) (2/3) {4};
\node[draw,circle] at (2,-3) (2/4) {3};
\node[draw,circle] at (2,-4) (2/5) {6};
\node[draw,circle] at (2,-5) (2/6) {5};

\node[draw,circle] at (3,0) (3/1) {2};
\node[draw,circle] at (3,-1) (3/2) {1};
\node[draw,circle] at (3,-2) (3/3) {4};
\node[draw,circle] at (3,-3) (3/4) {3};
\node[draw,circle] at (3,-4) (3/5) {6};
\node[draw,circle] at (3,-5) (3/6) {5};

\node[draw,circle] at (4,0) (4/1) {2};
\node[draw,circle] at (4,-1) (4/2) {4};
\node[draw,circle] at (4,-2) (4/3) {1};
\node[draw,circle] at (4,-3) (4/4) {3};
\node[draw,circle] at (4,-4) (4/5) {6};
\node[draw,circle] at (4,-5) (4/6) {5};

\node[draw,circle] at (5,0) (5/1) {2};
\node[draw,circle] at (5,-1) (5/2) {4};
\node[draw,circle] at (5,-2) (5/3) {1};
\node[draw,circle] at (5,-3) (5/4) {6};
\node[draw,circle] at (5,-4) (5/5) {3};
\node[draw,circle] at (5,-5) (5/6) {5};

\node[draw,circle] at (6,0) (6/1) {2};
\node[draw,circle] at (6,-1) (6/2) {4};
\node[draw,circle] at (6,-2) (6/3) {6};
\node[draw,circle] at (6,-3) (6/4) {1};
\node[draw,circle] at (6,-4) (6/5) {3};
\node[draw,circle] at (6,-5) (6/6) {5};

\node[draw,circle] at (7,0) (7/1) {2};
\node[draw,circle] at (7,-1) (7/2) {4};
\node[draw,circle] at (7,-2) (7/3) {6};
\node[draw,circle] at (7,-3) (7/4) {3};
\node[draw,circle] at (7,-4) (7/5) {1};
\node[draw,circle] at (7,-5) (7/6) {5};

\draw[ultra thick] (0/1) -- (1/1) -- (2/1) -- (3/2) -- (4/3) -- (5/3) -- (6/4) -- (7/5);
\draw[ultra thick] (0/2) -- (1/2) -- (2/2) -- (3/1) -- (4/1) -- (5/1) -- (6/1) -- (7/1);
\draw[ultra thick] (0/3) -- (1/3) -- (2/4) -- (3/4) -- (4/4) -- (5/5) -- (6/5) -- (7/4);
\draw[ultra thick] (0/4) -- (1/4) -- (2/3) -- (3/3) -- (4/2) -- (5/2) -- (6/2) -- (7/2);
\draw[ultra thick] (0/5) -- (1/6) -- (2/6) -- (3/6) -- (4/6) -- (5/6) -- (6/6) -- (7/6);
\draw[ultra thick] (0/6) -- (1/5) -- (2/5) -- (3/5) -- (4/5) -- (5/4) -- (6/3) -- (7/3);

\draw[ultra thick, dashed] (6/3) -- (7/4); 
\draw[ultra thick, dashed] (6/4) -- (7/3); 

\draw[ultra thick, dashed] (5/2) -- (6/3); 
\draw[ultra thick, dashed] (5/3) -- (6/2); 

\draw[ultra thick, dashed] (3/2) -- (4/1); 
\draw[ultra thick, dashed] (3/1) -- (4/2); 

\draw[ultra thick, dashed] (2/1) -- (2.75,.75); 
\draw[ultra thick, dashed] (2.25,.75) -- (3/1); 

\epic
\\
Wiring diagram for $w$ & Wiring diagram for $w{\uparrow_7}$ \\\\
\pic[scale=.85]
\node[draw,circle] at (0,0) (0/1) {1};
\node[draw,circle] at (0,-1) (0/2) {2};
\node[draw,circle] at (0,-2) (0/3) {3};
\node[draw,circle] at (0,-3) (0/4) {4};
\node[draw,circle] at (0,-4) (0/5) {5};
\node[draw,circle] at (0,-5) (0/6) {6};
\node[draw,circle] at (0,-6) (0/7) {7};

\node[draw,circle] at (1,0) (1/1) {1};
\node[draw,circle] at (1,-1) (1/2) {2};
\node[draw,circle] at (1,-2) (1/3) {3};
\node[draw,circle] at (1,-3) (1/4) {4};
\node[draw,circle] at (1,-4) (1/5) {5};
\node[draw,circle] at (1,-5) (1/6) {7};
\node[draw,circle] at (1,-6) (1/7) {6};

\node[draw,circle] at (2,0) (2/1) {1};
\node[draw,circle] at (2,-1) (2/2) {2};
\node[draw,circle] at (2,-2) (2/3) {3};
\node[draw,circle] at (2,-3) (2/4) {5};
\node[draw,circle] at (2,-4) (2/5) {4};
\node[draw,circle] at (2,-5) (2/6) {7};
\node[draw,circle] at (2,-6) (2/7) {6};

\node[draw,circle] at (3,0) (3/1) {2};
\node[draw,circle] at (3,-1) (3/2) {1};
\node[draw,circle] at (3,-2) (3/3) {3};
\node[draw,circle] at (3,-3) (3/4) {5};
\node[draw,circle] at (3,-4) (3/5) {4};
\node[draw,circle] at (3,-5) (3/6) {7};
\node[draw,circle] at (3,-6) (3/7) {6};

\node[draw,circle] at (4,0) (4/1) {2};
\node[draw,circle] at (4,-1) (4/2) {3};
\node[draw,circle] at (4,-2) (4/3) {1};
\node[draw,circle] at (4,-3) (4/4) {5};
\node[draw,circle] at (4,-4) (4/5) {4};
\node[draw,circle] at (4,-5) (4/6) {7};
\node[draw,circle] at (4,-6) (4/7) {6};

\node[draw,circle] at (5,0) (5/1) {2};
\node[draw,circle] at (5,-1) (5/2) {3};
\node[draw,circle] at (5,-2) (5/3) {1};
\node[draw,circle] at (5,-3) (5/4) {5};
\node[draw,circle] at (5,-4) (5/5) {7};
\node[draw,circle] at (5,-5) (5/6) {4};
\node[draw,circle] at (5,-6) (5/7) {6};

\node[draw,circle] at (6,0) (6/1) {2};
\node[draw,circle] at (6,-1) (6/2) {3};
\node[draw,circle] at (6,-2) (6/3) {5};
\node[draw,circle] at (6,-3) (6/4) {1};
\node[draw,circle] at (6,-4) (6/5) {7};
\node[draw,circle] at (6,-5) (6/6) {4};
\node[draw,circle] at (6,-6) (6/7) {6};

\node[draw,circle] at (7,0) (7/1) {2};
\node[draw,circle] at (7,-1) (7/2) {3};
\node[draw,circle] at (7,-2) (7/3) {5};
\node[draw,circle] at (7,-3) (7/4) {7};
\node[draw,circle] at (7,-4) (7/5) {1};
\node[draw,circle] at (7,-5) (7/6) {4};
\node[draw,circle] at (7,-6) (7/7) {6};

\draw[ultra thick] (0/1) -- (1/1) -- (2/1) -- (3/2) -- (4/3) -- (5/3) -- (6/4) -- (7/5);
\draw[ultra thick] (0/2) -- (1/2) -- (2/2) -- (3/1) -- (4/1) -- (5/1) -- (6/1) -- (7/1);
\draw[ultra thick] (0/3) -- (1/3) -- (2/3) -- (3/3) -- (4/2) -- (5/2) -- (6/2) -- (7/2);
\draw[ultra thick] (0/4) -- (1/4) -- (2/5) -- (3/5) -- (4/5) -- (5/6) -- (6/6) -- (7/6);
\draw[ultra thick] (0/5) -- (1/5) -- (2/4) -- (3/4) -- (4/4) -- (5/4) -- (6/3) -- (7/3);
\draw[ultra thick] (0/6) -- (1/7) -- (2/7) -- (3/7) -- (4/7) -- (5/7) -- (6/7) -- (7/7);
\draw[ultra thick] (0/7) -- (1/6) -- (2/6) -- (3/6) -- (4/6) -- (5/5) -- (6/5) -- (7/4);

\epic
&
\pic

\draw (3,0) -- (0,0) -- (0,-4);
\draw (0,-1) -- (1,-1) -- (1,0);
\node at (0.5,-0.5) (1/1) {1};
\draw (1,-1) -- (2,-1) -- (2,0);
\node at (1.5,-0.5) (2/1) {3};
\draw (2,-1) -- (3,-1) -- (3,0);
\node at (2.5,-0.5) (3/1) {7};
\draw (0,-2) -- (1,-2) -- (1,-1);
\node at (0.5,-1.5) (1/2) {2};
\draw (1,-2) -- (2,-2) -- (2,-1);
\node at (1.5,-1.5) (2/2) {6};
\draw (0,-3) -- (1,-3) -- (1,-2);
\node at (0.5,-2.5) (1/3) {4};
\draw (0,-4) -- (1,-4) -- (1,-3);
\node at (0.5,-3.5) (1/4) {5};

\node at (.5,.5) {1};
\node at (1.5,.5) {4};
\node at (2.5,.5) {6};

\node at (-.5,-.5) {7};
\node at (-.5,-1.5) {5};
\node at (-.5,-2.5) {3};
\node at (-.5,-3.5) {2};

\epic\\
Wiring diagram for $w{\uparrow_7 \uparrow_7}$ & $\text{Tab}(w{\uparrow_7 \uparrow_7})=\text{LS}(w)$
\end{tabular}
\end{figure}


\subsection{Grassmannian permutations}
\label{ssec:grassmannian}

Recall a permutation $\sigma$ is Grassmannian if it has exactly one descent. 
We can then write 
\[
\sigma = a_1a_2\dots a_{k}b_1b_2\dots b_{n-k}
\]
where $\{a_i\}^{k}_{i=1}$ and $\{b_j\}^{n-k}_{j=1}$ are increasing sequences with $a_k > b_1$. 
A word $w$ is \emph{Grassmannian} if it is the reduced word of a Grassmannian permutation.
From the Grassmannian word $w = w_1 \dots w_m$ we construct a tableau $\Tab(w)$ as follows.
Index the columns of $\Tab(w)$ by $b_1, \dots, b_{n-k}$ and the rows by $a_{k}, a_{k-1}, \dots, a_1$.
Since all inversions in $\sigma$ feature an $a_i$ and a $b_j$, each $w_l$ in $w$ represents the swap between an $a_i$ and a $b_j$.
For $w_l$, we enter $m+1-l$ in the column indexed by $a_i$ and $b_j$.
If $a_i$ swaps with $b_j$, we see it must later swap with each smaller $b$.
This shows entries are increasing across rows.
Likewise, if $b_j$ swaps with $a_i$, it must later swap with each larger $a$ so entries increase down columns.
From this, we can conclude that $\Tab(w)$ is a standard Young tableau whose shape is determined by $\sigma$.
There is an example of $\Tab$ in Figure~\ref{fig:little example}.
For a given Grassmannian permutation $\sigma$, this map is a bijection as the process is easily reversed.
Multiple Grassmannian permutations may correspond to the same shape.
However, they will only differ by some fixed points at the beginning and end of the permutation.

\subsection{Little bumps and the Little map}
\label{ssec:little} 

We now describe the method in~\cite{little2003combinatorial} for transforming an arbitrary reduced word into the reduced word of a Grassmannian permutation.   This method is implemented in an easy-to-use Java applet~\cite{littleapplet} and we strongly recommend that the reader follow along in our descriptions using this this software, if possible.
Let $w = w_1 \dots w_m$ be a reduced word and $w^{(i)} = w_1 \dots w_{i-1}w_{i+1}\dots w_m$.
We construct 
\[
w^{(i-)} = 
\begin{cases} w_1 \dots w_{i-1}(w_i-1)w_{i+1}\dots w_m &  \mbox{if } w_i > 1\\
(w_1+1) \dots (w_{i-1}+1)w_i(w_{i+1}+1)\dots (w_m +1) &  \mbox{if } w_i = 1 \end{cases}
\]
by decrementing $w_i$ by one or incrementing each other entry if $w_i =1$.

Let $w$ be a reduced word such that $w^{(i)}$ is also reduced.
Note $w^{(i-)}$ may not be reduced, as $w_i -1$ may swap the same values as some $w_j$ with $j\neq i$.
However, this is the only way $w^{(i)-}$ can fail to be reduced as $w^{(i)}$ is reduced and we have added only one additional swap.
Removing $w_j$ from $w^{(i-)}$, we obtain a new reduced word $w^{(i-)(j)}$.
Repeating this process of decrementation, we can construct $w^{(i-)(j-)}$ and so on until we are left with a reduced word $v = v_1 \dots v_m$.
We refer to this process as a  \emph{Little bump} beginning at position $i$ and say $v = w{\uparrow_i}$, where $i$ is the initial index the bump was started at.
To see that this process terminates, we refer to the following lemma.
\begin{lemma}[Lemma 5,~\cite{little2003combinatorial}]
\label{lem:decrement once}
Let $w$ be a reduced word such that $w^{(i)}$ is reduced. Let $i_1, i_2 , \dots$ be the sequence of indices decremented in $w{\uparrow_{i}}$. Then the entries of $i_1, i_2, \dots$ are unique.
\end{lemma}
Since $w$ is finite, we see the process terminates so that $w{\uparrow_{i}}$ is well-defined.
We highlight a property of Little bumps observed in~\cite{little2003combinatorial}, that they preserve the descent structure of $w$.
\begin{corollary}
\label{cor:descent structure}
Let $w = w_1 \dots w_m$ and $v = v_1 \dots v_m$ be a reduced words and $\uparrow$ be a Little bump such that $v = w{\uparrow}$.
Then $v_i > v_{i+1}$ if and only if $w_i > w_{i+1}$ for all $i$.
\end{corollary}
\begin{proof}
Let $w_i > w_{i+1}$. As each $w_i$ is decremented at most once, we see $v_i \geq v_{i+1}$, but $v_i \neq v_{i+1}$. Thus $v_i > v_{i+1}$. By the same reasoning, if $w_i < w_{i+1}$, we see $v_i<v_{i+1}$.
\end{proof}

Let $w$ be a reduced word of $\sigma \in S_n$.
We define the Little map $\LS(w)$.
\begin{enumerate}
\item If $w$ is a Grassmannian word, then $\LS(w) = \Tab(w)$
\item If $w$ is not a Grassmannian word, identify the swap location $i$ of the last inversion (lexicographically) in $\sigma$ and output $\LS(w{\uparrow_i})$.
\end{enumerate}
It is a result from~\cite{little2003combinatorial} extending work in~\cite{lascoux1985schubert} that $\LS$ terminates.
We then see that $w \mapsto \LS(w)$ where $\LS(w)$ is a standard Young tableau.
An example can be seen in Figure \ref{fig:little example}, where the word $w$ is represented by its \emph{wiring diagram}: an arrangement of horizontal, parallel wires spaced one unit apart, labelled $1$ through $n$ on the left-hand side, in which each letter in the word $w$ is represented by crossings of wires.

\section{The action of Coxeter-Knuth moves}
\label{sec:ck moves}

\subsection{Basics of Coxeter-Knuth moves}
\label{ssec:basics of ck}

First introduced in~\cite{edelman1987balanced}, Coxeter-Knuth moves are perhaps the most important tool for studying Edelman-Greene insertion.
They are modifications of the second and third Coxeter relations.
Let $a<b<c$ and $x$ be integers. The three \emph{Coxeter-Knuth moves} are the modifications
\begin{enumerate}
\item $acb \leftrightarrow cab$
\item $bac \leftrightarrow bca$
\item $x(x+1)x \leftrightarrow (x+1)x(x+1)$
\end{enumerate}
applied to three consecutive entries of a reduced word.
Let $w = w_1 w_2 \dots w_m$ be a reduced word of $\sigma$ and $\alpha$ denote a Coxeter-Knuth move on the entries $w_{i-1} w_i w_{i+1}$.
Since $a<b<c$,  if $\alpha$ is of type one or two we have $w\alpha$ a reduced word of $\sigma$ as well by the second Coxeter relation.
If $\alpha$ is of type three then $w\alpha$ is a reduced word of $\sigma$ by the third Coxeter relation. 
We say two reduced words $v$ and $w$ are \emph{Coxeter-Knuth equivalent} if there exists a sequence $\alpha_1, \alpha_2, \dots, \alpha_k$ of Coxeter-Knuth moves such that 
\[
v = w\alpha_1\dots\alpha_k.
\]
Note that two Coxeter-Knuth equivalent reduced words must correspond to reduced decompositions of the same permutation.
We can see their action on wiring diagrams in Figure \ref{fig:ck moves}.
\begin{figure}
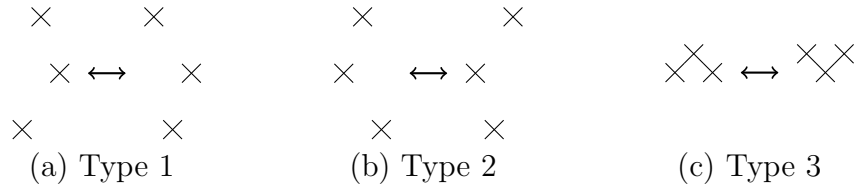

\caption{The three types of Coxeter-Knuth moves acting on wiring diagrams.
\label{fig:ck moves}}
\vspace{10pt}
\begin{tabular}{ccc}
\rule{0.25in}{0in}
\pic[scale =.25]

\draw (2,0) -- (3,1);
\draw (2,1) -- (3,0);

\draw (1,3) -- (2,4);
\draw  (1,4) -- (2,3);

\draw (0,-3) -- (1,-2);
\draw (0,-2) -- (1,-3);

\draw[thick, <->] (4,.5) -- (6,.5);

\draw (9,0) -- (10,1);
\draw (9,1) -- (10,0);

\draw (7,3) -- (8,4);
\draw  (7,4) -- (8,3);

\draw (8,-3) -- (9,-2);
\draw (8,-2) -- (9,-3);

\epic 
\rule{0.25in}{0in}
&
\rule{0.25in}{0in}
\pic[scale =.25]

\draw (0,0) -- (1,1);
\draw (0,1) -- (1,0);

\draw (1,3) -- (2,4);
\draw  (1,4) -- (2,3);

\draw (2,-3) -- (3,-2);
\draw (2,-2) -- (3,-3);

\draw[thick, <->] (4,.5) -- (6,.5);

\draw (7,0) -- (8,1);
\draw (7,1) -- (8,0);

\draw (9,3) -- (10,4);
\draw  (9,4) -- (10,3);

\draw (8,-3) -- (9,-2);
\draw (8,-2) -- (9,-3);

\epic 
\rule{0.25in}{0in}
&
\rule{0.25in}{0in}
\pic[scale =.25]

\node at (0,-2.5) {};

\draw (0,0) -- (1,1);
\draw (0,1) -- (1,0);

\draw (1,1) -- (2,2);
\draw  (1,2) -- (2,1);

\draw (2,0) -- (3,1);
\draw (2,1) -- (3,0);

\draw[thick, <->] (4,.5) -- (6,.5);

\draw (7,1) -- (8,2);
\draw (7,2) -- (8,1);

\draw (9,1) -- (10,2);
\draw  (9,2) -- (10,1);

\draw (8,0) -- (9,1);
\draw (8,1) -- (9,0);

\epic
\rule{0.25in}{0in}
\\ (a) Type 1 & (b) Type 2 & (c) Type 3
\end{tabular}
\end{figure}

Coxeter-Knuth moves play a role in the study of Edelman-Greene insertion analogous to that of Knuth moves in the study of RSK insertion.
\begin{theorem}[Theorem 6.24 in~\cite{edelman1987balanced}]
\label{thm:ck fix p}
Let $v$ and $w$ be a reduced words. 
Then $P(v)= P(w)$ if and only if $v$ and $w$ are Coxeter-Knuth equivalent.
\end{theorem}

\subsection{The action of Coxeter-Knuth moves on $Q(w)$}
\label{ssec:ck on q}

In order to understand the relationships of Coxeter-Knuth moves and Little bumps, we must first understand in greater detail how Coxeter-Knuth moves relate to Edelman-Greene insertion.
From Theorem \ref{thm:ck fix p}, we understand how Coxeter-Knuth moves relate to $P(w)$.
We must also understand their action on $Q(w)$.
For $T$ a standard Young tableau with $n$ entries, let $Tt_{i,j}$ be the Young tableau obtained by swapping the entries labeled $n-i$ and $n-j$.
\begin{lemma}
\label{lem:ck on q}
Let $w = w_1 \dots w_m$ be a reduced word and $\alpha$ be a Coxeter-Knuth move on $w_{i-1}w_iw_{i+1}$. 
If $\alpha$ is a Coxeter-Knuth move of type one or three, then
\[
Q(w\alpha) = Q(w)t_{i-1,i}.
\]
If $\alpha$ is a Coxeter-Knuth move of type two, then $\alpha$ acts on $Q(w)$ as above or
\[
Q(w\alpha) = Q(w)t_{i-1,i} \ \ \ \ \  \text{or} \ \ \ \ \ Q(w\alpha) = Q(w)t_{i,i+1}.
\]
\end{lemma}

\begin{proof}
For $w = w_1 \dots w_m$ a reduced word we see $w|_{i-1} := w_{i-1}w_i \dots w_m$ is also a reduced word. 
Let $\alpha$ be a Coxeter-Knuth move on $w_{i-1}w_iw_{i+1}$. 
By Theorem~\ref{thm:ck fix p} we see
\[
P(w|_{i-1})   = P(w|_{i-1}\alpha) = P(w\alpha|_{i-1})
\]
as they differ by a Coxeter-Knuth move.
Since $w_1 \dots w_{i-2}$ remain unmodified, their insertion is unchanged.
Additionally, as $w\mid_{i+2} = w\alpha|_{i+2}$ we see $Q(w|_{i+2}) = Q(w\alpha|_{i+2})$, so changes in $Q(w)$ can only occur at the entries labeled $i{-}1,\  i$ and $i{+}1$.
The remainder of this argument is adapted from the proof for Theorem 6.24 in~\cite{edelman1987balanced}.
The strategy of proof is to analyze the insertion of the triplet $w_{i-1}w_iw_{i+1}$ and its counterpart in $w\alpha$ into one row of $P(w|_{i+2})$.
If one such $w_j$ fails to bump anything, the analysis is straightforward.
Otherwise the three entries bumped by each will continue to differ by a Coxeter-Knuth move, allowing us to reduce the problem to the previous case.

\begin{enumerate}
\item 
In this case, we do not need to complete the full analysis described above.
Let $\alpha$ be a Coxeter-Knuth move of type one. 
Then $w_{i+1}$ inserts into the same spot in $P(w|_{i+2})$ for both $w$ and $w\alpha$.
Since $Q(w) \neq Q(w\alpha)$, we see $Q(w\alpha) = Q(w)t_{i-1,i}$.

\item Let $\alpha$ be a Coxeter-Knuth move of type three. 
This case is treated now as the case where $\alpha$ is a move of type two depends on it. 
We compare the insertion of $x(x+1)x$ and $(x+1)x(x+1)$ into the same row of $P(w|_{i+2})$. 
Assume both $x$ and $x+1$ bump an entry of the row. 
Let $p$ denote the entry bumped by $x$, $\epsilon_1$ be the entry preceding $p$ and $\epsilon_2$ be the entry following $p$. 
If $p > x+1$, we see $x$ and $x+1$ are inserted into the same position, so $Q(w\alpha) = Q(w)t_{i-1,i}$ as in the first case. 

Let $p = x+1$. 
Since $w|_{i+2}$ is reduced, $\epsilon_2 = x+2$ (otherwise, inserting $x+1$ first would leave consecutive occurrences of $x+1$). 
There are two remaining possibilities: $\epsilon_1<x$ or $\epsilon_1=x$. 
Let $\epsilon_1<x$.
Upon inserting $x(x+1)x$ into the row, we see the first $x$ bumps $x+1$, $x+1$ bumps $x+2$ and the second $x$ bumps the $x+1$ just inserted, so that $(x+1)(x+2)(x+1)$ is inserted into the next row. 
Upon inserting $(x+1)x(x+1)$ into the row, we see the first $x+1$ produces a special bump of $x+2$, the $x$ bumps $x+1$ and the second $x+1$ bumps the $x+2$ remaining after the special bump, so that $(x+2)(x+1)(x+2)$ is inserted into the next row. 
The case where $\epsilon_1 =x$ is simpler.
Every bump is a special bump, so that $p$ and $\epsilon_2$ are unchanged throughout the insertion process.
Each $x$ and $x+1$ will bump an entry precisely one larger, so that the entries to be inserted into the next row will be $(x+1)(x+2)(x+1)$ and $(x+2)(x+1)(x+2)$ respectively.
In both cases, we are left with a Coxeter-Knuth move of type three.

If one of the three inserted letters does not bump an entry of the row, we see the largest entry $k$ of the row must be less than $x+1$. 
As $P(w\alpha |_{i+1})$ is row and column strict, we see $k < x$, so $x$ or $x+1$ would both insert at the end of the row. 
Thus $Q(w\alpha) = Q(w) t_{i-1,i}$.

\item Let $\alpha$ be a Coxeter-Knuth move of type two. 
For $a<b<c$, we compare the insertion of $bca$ and $bac$ into a row of $P(w|_{i+2})$ bumping $pqr$ and $p'q'r'$ respectively.
If $p=p'$, we see $Q(w \alpha) = Q(w)t_{i-1,i}$ as as in the first case. 
Assume $p \neq p'$. 
One can then check that $a$ and $c$ bump the same entries regardless of order, so that $p = q'$ and $q = p'$.
If $b$ bumps the same entry in each case, it is straightforward to see that $pqr$ and $p'q'r'$ differ by a Coxeter-Knuth move of type two.
The only way this does not occur is if $p = q + 1$ (so they are next to each other) with $c = q$.
In this case, upon inserting $bca$ we obtain $q(q+1)q$, while the insertion of $bac$ produces $(q+1)q(q+1)$ as the bump of $q+1$ by $c$ is special.
Therefore, we are left with another Coxeter-Knuth move of type two or one of type three.
In the latter case, we see $Q(w \alpha) = Q(w)t_{i-1,i}$ by the second case.

If some letter does not bump an entry of the row, there are two possibilities. 
Let $k$ be the largest entry of the row. If $k < a$, then $a$ and $c$ would insert into the same position, so $Q(w\alpha) = Q(w)t_{i-1,i}$. 
If $a < k <c$, then $c$ inserts on the end of the row and $a$ bumps the same entry $x$ of the row regardless of the order of insertion.
Since $x$ is the only entry bumped, we see $P(w|_i) = P(w\alpha|_i)$. 
Therefore, $Q(w\alpha) = Q(w)t_{i,i+1}$.

\end{enumerate}

\end{proof}

\subsection{Coxeter-Knuth moves and Little bumps}

We now set out to show that Coxeter-Knuth moves commute with Little bumps.
This requires two results.
The first  is that the order we perform a Coxeter-Knuth move $\alpha$ and a Little bump $\uparrow$ does not affect the resulting reduced word.
\begin{lemma}
\label{lem:ck little commute}
Let $w = w_1\dots w_m$ be a reduced word, $\alpha$ a Coxeter-Knuth move on $w_{i-1}w_iw_{i+1}$, and $\uparrow_{j,k}$ be a Little bump begun at the swap between the $j$ and $k$th trajectories.
Then
\[
(w\alpha){\uparrow_{j,k}} = (w{\uparrow_{j,k}})\alpha.
\]
\end{lemma}
Note that $\alpha$ is used here to represent two Coxeter-Knuth moves, possibly of different types, on the same indices.

\begin{proof}
Let $v = w {\uparrow_{j,k}}$  and $v' = (w\alpha){\uparrow_{j,k}}$.
Recall from Lemma \ref{lem:decrement once} and Corollary \ref{cor:descent structure} that $w_j-v_j\in\{0,1\}$ and $v$ has the same descent structure of $w$.

\begin{enumerate}

\item Let $\alpha$ be a Coxeter-Knuth move of the first type, i.e. $w_{i-1} w_i w_{i+1} \mapsto w_i w_{i-1} w_{i+1}$ with $w_{i+1}$ strictly between $w_{i-1}$ and $w_i$. 
Since a Little bump decrements an entry of $w$ by at most one, one can check that if $w_{i+1}$ differs from $w_i$ or $w_{i-1}$ by more than one, we can perform a Coxeter-Knuth move of type one on $v_{i-1}v_iv_{i+1}$.
In the event that they differ by exactly one and the smallest entry is decremented, we see in Figure~\ref{fig:transitional bumps 1} that after the bump they differ by a Coxeter-Knuth move of the third type.

\item  Let $\alpha$ be a Coxeter-Knuth move of the second type, i.e. $w_{i-1} w_i w_{i+1} \mapsto w_{i-1} w_{i+1} w_i$ with $w_{i-1}$ strictly between $w_{i+1}$ and $w_i$. 
Since a Little bump decrements an entry of $w$ by at most one, one can check that if $w_{i-1}$ differs from $w_i$ or $w_{i+1}$ by more than one, there is a Coxeter-Knuth move of type two on $v_{i-1}v_iv_{i+1}$.
In the event that they differ by exactly one and the smallest entry is bumped, we see in Figure~\ref{fig:transitional bumps 2} that after the bump they differ by a Coxeter-Knuth move of the third type.

\item Let $\alpha$ be a Coxeter-Knuth move of the third type. Note the middle entry cannot be bumped unless all three entries are bumped. In the event fewer entries (but not zero) are bumped, we see in Figure~\ref{fig:transitional bumps 2} that there will be a Coxeter-Knuth move of the first or second type remaining.

We next show that the rest of the Little bump proceeds in the same manner once the crossings involved in the Coxeter-Knuth move have been bumped. 
To see this, we need only observe that the last bumped swap is between the same two trajectories. This can be verified readily by examining Figures~\ref{fig:transitional bumps 1} and~\ref{fig:transitional bumps 2}.

The preceding argument assumes that the bumping path does not return to the crossings involved in the Coxeter-Knuth move.  
It is possible that the bumping path passes through the crossings involved in the Coxeter-Knuth path twice (but no more than that, by Lemma~\ref{lem:decrement once}).  
However, the same argument applies, showing that all three crossings are bumped regardless of whether the Coxeter-Knuth move is performed before or after the bump.

\end{enumerate}

\end{proof}

\begin{figure}
\caption{Transitional bumps for type one and two Coxeter-Knuth moves.
\label{fig:transitional bumps 1}}
\vspace{10pt}
\centering
\pic[scale =.5]

\draw (0,0) -- (1,1) -- (2,1) -- (3,2);
\draw (0,1) -- (1,0) -- (3,0);
\draw (0,2) -- (1,2) -- (2,3) -- (3,3);
\draw (0,3) -- (1,3) -- (3,1);

\draw (.5,.5) circle (0.5);

\draw[thick, ->] (4,1.5) -- (5,1.5) node[above, midway] {$\uparrow$};

\draw (6,0) -- (9,0);
\draw (6,1) -- (8,3) -- (9,3);
\draw (6,2) -- (7,1) -- (8,1) -- (9,2);
\draw (6,3) -- (7,3) -- (9,1);

\draw (6.5,1.5) circle (0.5);

\draw[thick, <->] (1.5,-1) -- (1.5,-2) node[midway, left] {$\alpha$};

\draw (0,-3) -- (1,-4) -- (2,-4) -- (3,-5);
\draw (0,-4) -- (1,-3) -- (3,-3);
\draw (0,-5) -- (1,-5) -- (2,-6) -- (3,-6);
\draw (0,-6) -- (1,-6) -- (3,-4);

\draw (1.5,-5.5) circle (0.5);
\draw (2.5,-4.5) circle (0.5);

\draw[thick, ->] (4,-4.5) -- (5,-4.5) node[above, midway] {$\uparrow$};

\draw (6,-3) -- (8,-5) -- (9,-5);
\draw (6,-4) -- (7,-3)-- (8,-3) -- (9,-4);
\draw (6,-5) -- (7,-5) -- (9,-3);
\draw (6,-6) -- (9,-6);

\draw (8.5,-3.5) circle (0.5);

\epic
\hspace{36pt}
\pic[scale =.5]

\draw (0,0) -- (2,0) -- (3,1);
\draw (0,1) -- (2,3) -- (3,3);
\draw (0,2) -- (1,1) -- (2,1) -- (3,0);
\draw (0,3) -- (1,3) -- (2,2) -- (3,2);

\draw (2.5,.5) circle (0.5);

\draw[thick, ->] (4,1.5) -- (5,1.5) node[above, midway] {$\uparrow$};

\draw (6,0) -- (9,0);
\draw (6,1) -- (8,3) -- (9,3);
\draw (6,2) -- (7,1) -- (8,1) -- (9,2);
\draw (6,3) -- (7,3) -- (9,1);

\draw (8.5,1.5) circle (0.5);

\draw[thick, <->] (1.5,-1) -- (1.5,-2) node[midway, left] {$\alpha$};

\draw (0,-3) -- (2,-3) -- (3,-4);
\draw (0,-4) -- (2,-6) -- (3,-6);
\draw (0,-5) -- (1,-4) -- (2,-4) -- (3,-3);
\draw (0,-6) -- (1,-6) -- (2,-5) -- (3,-5);

\draw (1.5,-5.5) circle (0.5);
\draw (.5,-4.5) circle (0.5);

\draw[thick, ->] (4,-4.5) -- (5,-4.5) node[above, midway] {$\uparrow$};

\draw (6,-3) -- (8,-5) -- (9,-5);
\draw (6,-4) -- (7,-3)-- (8,-3) -- (9,-4);
\draw (6,-5) -- (7,-5) -- (9,-3);
\draw (6,-6) -- (9,-6);

\draw (6.5,-3.5) circle (0.5);

\epic
\end{figure}
\begin{figure}
\caption{Transitional bumps for type three Coxeter-Knuth moves
\label{fig:transitional bumps 2}}
\vspace{10pt}
\pic[scale =.5]
\draw (0,0) -- (1,0) -- (3,2);
\draw (0,1) -- (1,2) -- (2,2) -- (3,1);
\draw (0,2) -- (2,0) -- (3,0);
\draw (0,3) -- (3,3);

\draw (.5,1.5) circle (0.5);

\draw[thick, ->] (4,1.5) -- (5,1.5) node[above, midway] {$\uparrow$};

\draw (6,0) -- (7,0) -- (9,2);
\draw (6,1) -- (7,1) -- (8,0) -- (9,0);
\draw (6,2) -- (7,3) -- (9,3);
\draw (6,3) -- (7,2) -- (8,2) -- (9,1);

\draw (6.5,2.5) circle (0.5);

\draw[thick, <->] (1.5,-1) -- (1.5,-2) node[midway, left] {$\alpha$};

\draw (0,-3) --  (3,-3);
\draw (0,-4) -- (1,-4) -- (3,-6);
\draw (0,-5) -- (1,-6) -- (2,-6) -- (3,-5);
\draw (0,-6) -- (2,-4) -- (3,-4);

\draw (1.5,-4.5) circle (0.5);
\draw (2.5,-5.5) circle (0.5);

\draw[thick, ->] (4,-4.5) -- (5,-4.5) node[above, midway] {$\uparrow$};

\draw (6,-3) -- (7,-3) -- (9,-5);
\draw (6,-4) -- (7,-4) -- (8,-3) -- (9,-3);
\draw (6,-5) -- (7,-6) -- (9,-6);
\draw (6,-6) -- (7,-5) -- (8,-5) -- (9,-4);

\draw (7.5,-3.5) circle (0.5);

\epic
\hspace{36pt}
\pic[scale =.5]
\draw (0,0) -- (1,0) -- (3,2);
\draw (0,1) -- (1,2) -- (2,2) -- (3,1);
\draw (0,2) -- (2,0) -- (3,0);
\draw (0,3) -- (3,3);

\draw (2.5,1.5) circle (0.5);

\draw[thick, ->] (4,1.5) -- (5,1.5) node[above, midway] {$\uparrow$};

\draw (6,0) -- (7,0) -- (8,1) -- (9,1);
\draw (6,1) -- (7,2) -- (8,2) -- (9,3);
\draw (6,2) -- (8,0) -- (9,0);
\draw (6,3) -- (8,3) -- (9,2);

\draw (8.5,2.5) circle (0.5);

\draw[thick, <->] (1.5,-1) -- (1.5,-2) node[midway, left] {$\alpha$};

\draw (0,-3) --  (3,-3);
\draw (0,-4) -- (1,-4) -- (3,-6);
\draw (0,-5) -- (1,-6) -- (2,-6) -- (3,-5);
\draw (0,-6) -- (2,-4) -- (3,-4);

\draw (1.5,-4.5) circle (0.5);
\draw (.5,-5.5) circle (0.5);

\draw[thick, ->] (4,-4.5) -- (5,-4.5) node[above, midway] {$\uparrow$};

\draw (6,-3) -- (7,-3) -- (8,-4) -- (9,-4);
\draw (6,-4) -- (7,-5)-- (8,-5) -- (9,-6);
\draw (6,-5) -- (8,-3) -- (9,-3);
\draw (6,-6) -- (8,-6) -- (9,-5);

\draw (7.5,-3.5) circle (0.5);

\epic
\end{figure}

We now show that the action of a Coxeter-Knuth move on $Q(w)$ remains the same after applying a Little bump.
Combined with Lemma~\ref{lem:ck little commute}, this shows that the order in which Coxeter-Knuth moves and Little bumps are performed on a reduced word $w$ does not affect either the resulting reduced word or the resulting recording tableau.
\begin{lemma}
\label{lem:ck little on q}
Let $w$ be a reduced word, $\alpha$ be a Coxeter-Knuth move and $\uparrow$ a Little bump. 
Then $Q(w\alpha) = Q(w)t_{i,i+1}$ if and only if $Q(w{\uparrow} \alpha) = Q(w{\uparrow})t_{i,i+1}$.

\end{lemma}

\begin{proof}
By Lemma~\ref{lem:ck on q}, we see $\alpha$ must exchange $w_{i-1}w_iw_{i+1}$ or $w_iw_{i+1}w_{i+2}$. 
We show the result in the case where $\alpha$ is a Coxeter-Knuth move on $w_{i-1}w_iw_{i+1}$, so that $\alpha$ is a Coxeter-Knuth move of type two. 
The other outcome then follows.

Let $w' = w\alpha$. Then $w|_{i} = w_i w_{i+1} w_{i+2} \dots w_n$ and $w'|_i = w_{i+1} w_i w_{i+2} \dots w_n$ are the parts of $w$ and $w'$ respectively to the right of $w_{i-1}$. 
Applying Edelman-Greene insertion to $w|_i$ and $w'|_i$, we see $P(w|_i) = P(w'|_i)$ and $Q(w|_i) = Q(w'|_i) t_{i, i+1}$. 
Therefore, there exists a sequence of Coxeter-Knuth moves $\alpha_1 \dots \alpha_m$ such that $ w|_i = w'|_i \alpha_1 \dots \alpha_m$. 
We then see 
\[
Q(w {\uparrow}|_i) = Q((w' \alpha_1 \dots \alpha_m){\uparrow} |_i) = Q((w' {\uparrow}) \alpha_1 \dots \alpha_m |_i)
\]
by Lemma~\ref{lem:ck little commute}.
Therefore $w {\uparrow}|_i$ and $w' {\uparrow}|_i$ differ solely at their first two positions and are Coxeter-Knuth equivalent, so we see $Q(w {\uparrow}|_i)$ and $Q(w'{\uparrow}|_i)$ have the same shape with $Q(w {\uparrow} |_i) = Q(w' {\uparrow} |_i) t_{i,i+1}$. 
Thus $Q(w{\uparrow})$ and $Q(w' {\uparrow})$ vary in the same way as $Q(w)$ and $Q(w')$. 

Since the inverse of a Little bump is a Little bump of the upside down word, where all Coxeter-Knuth move types are preserved, the converse holds as well. 
Therefore $Q(w\alpha) = Q(w)t_{i,i+1}$ if and only if $Q(w{\uparrow} \alpha) = Q(w{\uparrow})t_{i,i+1}$.

\end{proof}

\section{Proof of Results}
\label{sec:results}

\subsection{The Grassmannian case}
\label{ssec:grassmannian equality}

Before proving Theorem~\ref{thm:same map}, we need to establish the base case where $w$ is a Grassmannian word.
In order to do so, we must understand which entries are exchanging places with each swap. 
For $w = w_1 \dots w_m$ a reduced word, we define $\sigma^i = s_{w_1} s_{w_2} \dots s_{w_i}$ where $\sigma^0$ is the identity permutation.
The $k$th \emph{trajectory} of $w$ is the sequence $\{\sigma^i(k)\}^m_{i=0}$.
For $w$ a Grassmannian word of $\sigma = a_1a_2\dots a_{k}b_1b_2\dots b_{n-k}$, observe that the $j$th column of $\Tab(w)$ lists the times for all swaps featuring $b_j$. 
Since all such swaps increase the value of $b_j$, we can reconstruct its trajectory from the number and location of these swaps.
Similarly, we can reconstruct the trajectory of each $a_i$ from the $k+1-i$th row of $\Tab(w)$.
We will find it convenient to identify the $k$th trajectory of a Grassmannian word with the indices $\{i_1,i_2, \dots, i_{t_k}\} \subset [n]$ of the swaps featuring $k$.
Since insertion takes place from right to left, we index the entries such that $i_1>i_2>\dots>i_{t_k}$.
\begin{lemma}
\label{lem:grassmannian same map}
Let $w = w_1 \dots w_m$ be a reduced decomposition of a Grassmannian permutation $\sigma$.
Then $\Tab(w) = Q(w)$.

\end{lemma}

\begin{proof}
Let $\sigma = a_1 a_2 \dots a_{n}b_1b_2\dots b_{n-k}$ be a Grassmannian permutation with sole descent $a_{k}b_1$ and $w = w_1 \dots w_m$ a reduced decomposition of $\sigma$.
Note the trajectories of the $b_j$'s are non-intersecting as no two swap with each other.

We now show that when applying Edelman-Greene insertion to $w$, if $w_k$ is in the trajectory of $b_j$, then $w_k$ will be inserted into the $j$th column of $P_{n+1-k}(w)$ and each entry bumped during this insertion will in turn insert into the $j$th column.
From this and the definition of $\Tab$, we can conclude that $\Tab(w) = Q(w)$.

If $b_1$ has the only non-trivial trajectory amongst the $b_j$, then $Q(w) = \Tab(w)$ trivially; there is only one column in $\Tab(w)$. 
Assume there are multiple $b_j$ with non-trivial trajectories. 
Let $\{i_1,i_2,\dots, i_{t_2}\}$ be the trajectory of $b_2$.
Note $w_{i_k} = w_{i_{k+1}}+1$.
Then $b_1$ has trajectory $\{l_1, \dots, l_{t_1}\}$ with $t_1 \geq t_2$ and $l_{k} > i_{k}$, {\it i.e.} the $k$th from last swap featuring $b_1$ comes later than the $k$th from last featuring $b_2$ and so on.
Inserting into from right to left, we see that upon inserting any $w_{i_j}$, we will have already inserted $w_{l_j}$.
Therefore, $w_{i_1}$ will be inserted into the second column as any previously inserted entry will be from the trajectory of $b_1$, and thus have inserted into the first column.
When $w_{i_2}$ is inserted, it too will insert into the second column as $w_{l_2}$ will have been inserted into the first  column.
For identical reasons as before, $w_{i_1}$ will remain in the second column upon being bumped.
We then see inductively that, unimpeded by other swaps, the trajectory of $b_2$ will insert one after another into the second column.
The same argument applies to $b_3$ and so on.
Thus $\Tab(w) = Q(w)$.
\end{proof}


%
%

\subsection{The column reading word}
\label{ssec:column word}

The only ingredient missing from our argument is a canonical form that is invariant under Little bumps.
\begin{definition}
\label{def:column word}
For $T$ a Young tableau with columns $C^1, C^2 \dots , C^m$ where $C^i = c^i_1, c^i_2, \dots, c^i_k$ with $c^i_j$ being the $(j,i)$th entry of $T$, we define the \emph{column reading word} of $T$ to be the word $\tau(T) = C^m C^{m-1} \dots C^1$. 
Note if $T$ is row and column strict then $P(\tau(T)) = T$ and each column of $Q(\tau(T))$ has consecutive entries.
For $w$ a reduced word, we define $\tau(w)$ to be  $\tau(P(w))$. 
By the previous observation, $w$ and $\tau(w)$ are Coxeter-Knuth equivalent.

\end{definition}
One can think of the column reading word as closely related to the bottom-up reading word.
Since insertion takes place from right to left, the column reading word is in some sense its transpose.

\begin{lemma}
\label{lem:column word invariant}
Let $w$ be a reduced word and $\uparrow$ a Little bump on $w$. Then
\[
Q(\tau(w)) = Q(\tau(w){\uparrow}).
\]

\end{lemma}

\begin{proof}
Let $w$ be a reduced word, $\tau(w) = C^m C^{m-1} \dots C^1$  and $\tau(w) {\uparrow} = D^m D^{m-1} \dots D^1$ (note $D^k$ is not {\it a priori} a column of $P(\tau(w){\uparrow})$). 
Since $\tau(w)$ and $\tau(w){\uparrow}$ have the same descent structure, we see $C^1$ and $D^1$ insert identically. 
As each entry of $\tau(w) {\uparrow}$ is decremented at most once and $P(\tau(w))$ is row and column strict, we see
\[
d^k_i \leq c^k_i \leq d^k_i+1 \leq d^{k+1}_i,
\]
so $d^{k+1}_i$ will not bump any $d^k_j$ with $j \leq i$. 
Therefore, any entry of $D^{k}$ will stay in the $k$th column of $P(\tau(w){\uparrow})$ for all $k$, that is the entries of the $k$th column of $P(\tau(w){\uparrow})$ are $D^k$. 
Thus $\tau(w) {\uparrow}$ is a column reading word with identical column sizes, so $Q(\tau(w)) = Q(\tau(w) {\uparrow})$.

\end{proof}

\subsection{Proof of Theorem~\ref{thm:same map} and its corollaries}
\label{ssec:proof of thm}

Combining Lemma~\ref{lem:column word invariant} with Lemmas~\ref{lem:ck little commute} and~\ref{lem:ck little on q}, we can conclude the following:
\begin{theorem}
\label{thm:q invariant under bump}
Let $w$ be a reduced word and $\uparrow$ be a Little bump on $w$. Then
\[
Q(w) = Q(w{\uparrow}).
\]

\end{theorem}

\begin{proof}
Let $w$ be a reduced word. 
There exists a sequence $\alpha_1, \alpha_2, \dots, \alpha_k$ of Coxeter-Knuth moves such that $w = \tau(w)\alpha_1 \dots \alpha_k$.
As $Q(\tau(w)) = Q(\tau(w){\uparrow})$ by Lemma~\ref{lem:column word invariant}, we compute
\begin{align*}
Q(w) & = Q(\tau(w)\alpha_1\dots \alpha_k)\\
 & = Q((\tau(w){\uparrow})\alpha_1\dots \alpha_k)\\
& = Q((\tau(w)\alpha_1 \dots \alpha_k){\uparrow}) = Q(w{\uparrow})
\end{align*}
where the third equality follows by Lemmas~\ref{lem:ck little commute} and~\ref{lem:ck little on q}.
\end{proof}

\begin{proof}[Proof of Theorem~\ref{thm:same map}]
Let $w$ be a reduced word and ${\uparrow_{i_1}},\dots,{\uparrow_{i_k}}$ be the sequence of canonical Little bumps.
By Theorem~\ref{thm:q invariant under bump} and Lemma~\ref{lem:grassmannian same map}, we see
\[
Q(w) = Q(w{\uparrow_{i_1}}\dots{\uparrow_{i_k}}) = \Tab(w{\uparrow_{i_1}}\dots{\uparrow_{i_k}}) = \LS(w).
\]

\end{proof}

We now demonstrate several corollaries, including Lam's Conjecture.
The first is a simple consequence of Corollary~\ref{cor:descent structure}.
\begin{corollary}
The descent structure of a reduced word $w$ is determined by $Q(w)$.
\end{corollary}
There is an analogous result for the Robinson-Schensted-Knuth algorithm, which appears first in a paper of Sch\"utzenberger~\cite{schutzenberger1963}.  It was subsequently rediscovered by Foulkes~\cite{foulkes1976enumeration}.  The proof can also be found in the standard reference\cite[Lemma 7.23.1]{stanley2001enumerative}.
For sorting networks, this result was proved in~\cite{edelman1987balanced}.

The next is Conjecture 11 from~\cite{little2005factorization}, which first appeared as Conjecture 4.3.3 in the appendix of~\cite{garsia2002saga}.
\begin{corollary}
Let $w$ be a reduced word and let ${\uparrow_{i_1}}, {\uparrow_{i_2}}, \dots, {\uparrow_{i_m}}$ be any sequence of Little bumps such that 
\[
v = w{\uparrow_{i_1}} \dots {\uparrow_{i_m}}
\]
is a Grassmannian word. Then $\Tab(v) = \LS(w)$.

\end{corollary}

This follows from Theorem~\ref{thm:q invariant under bump}. We can extend this result further.
Let $\lambda$ be a partition with $w$ a Grassmannian word whose corresponding tableau is of shape $\lambda$. 
The permutation $\sigma $ associated to $w$ can be characterized by the number of initial fixed points and terminal fixed points. 
A Grassmannian permutation is \emph{minimal} if it has no initial or terminal fixed points. 
Note the minimal Grassmannian permutation of a given shape is unique. Recall two reduced words \emph{communicate} if there exists a sequence of Little bumps and inverse Little bumps changing one to the other.

\begin{proof}[Proof of Theorem~\ref{thm:lam}]
Let $v$ and $w$ be reduced words.
Suppose first that $v$ and $w$ communicate. 
Then by Theorem \ref{thm:q invariant under bump}, we have that $Q(v) = Q(w)$.

Conversely, suppose that $Q(v)=Q(w)$.  
By applying the Little map, $w$ can be changed to the Grassmannian word $w'$ and $v$ to the Grassmannian word $v'$ by a sequence of Little bumps. 
Since $Q(w) = Q(w')$ and $Q(v) = Q(v')$, we can conclude that $v$ and $w$ communicate if Grassmannian permutations of the same shape communicate. 
To do this, we demonstrate a sequence of Little bumps that adds a fixed point at the end of an arbitrary Grassmannian permutation, and another sequence that converts a fixed point at the beginning into one at the end. 
By converting any fixed points at the beginning into ones at the end, then removing those at the end via inverse bumps, we get the minimal Grassmannian permutation of that shape.
Therefore, any Grassmannian permutation communicates with the minimal permutation of that shape. 
From this, we can conclude any two Grassmannian permutations with the same shape communicate.

We now construct our sequence of Little bumps. 
Let $\sigma = a_1 \dots a_kb_1\dots b_{n-k}$ be a Grassmannian permutation with $a_{k}b_1$ its sole descent. 
Start a Little bump at the last swap featuring each $b_j$, beginning with $b_1$, so that the first bump begins between $b_1$ and $a_{k}$.
We will show this sequence of bumps decrements every entry in each trajectory exactly once. 
This is equivalent to decrementing each entry of $w$. 
If $\sigma$ has initial fixed points, this will remove one of them, leaving a fixed point at the end. 
If $\sigma$ has no initial fixed point, this will leave $w$ the same but add a fixed point to the end of $\sigma$.

We now verify that our sequence works as described. 
First, we must verify that the swap locations at which we begin a Little bump are valid choices, that is that removing that swap from $w$ leaves a reduced word. 
To see this, note that the first such swap chosen is the swap between $a_{k}$ and $b_1$, the last swap in $w$. 
This bump will decrement every entry in the trajectory of $b_1$. 
After the first Little bump, the second swap chosen is the last in the trajectory of $b_2$. 
Since the trajectories of all $b_j$ with $j>2$ are unaffected by the initial Little bump, this is the last swap for both $b_2$ and $a_{k}$, so removing it leaves a reduced word. 
This bump will decrement every entry in the trajectory of $b_2$. 
Note because we have already decremented the swaps in the trajectory of $b_1$ and these trajectories were initially disjoint, they will remain disjoint after the second Little bump. Applying this line of reasoning inductively, we see that each Little bump in the sequence is a valid Little bump which decrements every entry of each trajectory. We have now shown $v$ and $w$ communicate if $Q(v) = Q(w)$.

\end{proof}

Additionally, we show how to embed Robinson-Schensted insertion and RSK in the Little map.
In doing so, we recover the main results of~\cite{little2005factorization} through a much simplified argument.
This embedding was first predicted as Conjecture 4.3.1 in the appendix of~\cite{garsia2002saga}.

\begin{theorem}
\label{thm:little rs}
Let $\sigma = \sigma_1 \dots \sigma_n \in S_n$, so that $w(\sigma) = (2\sigma_n - 1) \dots (2\sigma_1 -1)$ is a reduced word as it has no repeated entries.
Let $\mbox{RS}(\sigma) = (P'(\sigma),Q'(\sigma))$ be the output of Robinson-Schensted insertion applied to $\sigma$.
Upon applying the transformation $k \mapsto (k+1)/2$ to the entries of $\LS(w(\sigma))$, we obtain $Q'(\sigma)$. We can obtain $P'(\sigma)$ by applying the same transformation to $\LS(w(\sigma^{-1}))$. 

\end{theorem}

\begin{proof}
Since $\LS(w) = Q(w)$ and there are no special bumps, Edelman-Greene insertion will perform the same insertion process on $w$ as Robinson-Schensted insertion performs on $\sigma$.
Therefore, upon applying the transformation $k \mapsto (k+1)/2$, we see $\LS(w(\sigma)) = Q(w(\sigma)) = Q'(\sigma)$. 
Since $\mbox{RS}(\sigma^{-1}) = (Q'(\sigma),P'(\sigma))$ (see {\it e.g.}~\cite{stanley2001enumerative}), we can obtain $P'(\sigma)$ by applying the same transformation to $\LS(w(\sigma^{-1}))$.

\end{proof}
Since RSK can be embedded in Robinson-Schensted insertion (see Section 7 of~\cite{little2005factorization} for a description of this process), Theorem~\ref{thm:little rs} recovers an embedding of RSK into the Little map as well. 

\bibliographystyle{plain}
\bibliography{eg-little.bib}

\end{document}